\documentclass[11pt, english]{amsart}
\usepackage[utf8]{inputenc}

\usepackage{xcolor}
\usepackage{newtxmath}

\theoremstyle{plain}
\newtheorem{theorem}{Theorem}
\newtheorem*{theorem*}{Theorem}
\newtheorem*{conjecture*}{Conjecture}
\newtheorem{prop}[theorem]{Proposition}
\newtheorem{lemma}[theorem]{Lemma}

\def\CC{{\mathbb{C}}}

\def\PP{{\mathbb{P}}}
\def\QQ{{\mathbb{Q}}}\def\ZZ{{\mathbb{Z}}}

\def\cO{{\mathcal{O}}}

\def\lra{{\longrightarrow}}

\DeclareMathOperator{\pt}{pt}

\begin{document} 

\title[Quantum cohomology and irrationality of GM-fourfolds]{Quantum cohomology and irrationality of Gushel-Mukai fourfolds}
\author{Vladimiro Benedetti, Laurent Manivel, Nicolas Perrin}
\date{\today}

\address{
Universit\'e C\^ote d'Azur, CNRS, Laboratoire J.-A. Dieudonn\'e, Parc Valrose, F-06108 Nice Cedex 2, France} 
\email{vladimiro.benedetti@univ-cotedazur.fr}

\address{
Institut de Math\'ematiques de Marseille, 
Aix-Marseille University, CNRS, I2M, UMR 7373, Marseille, France}
\email{laurent.manivel@math.cnrs.fr}

\address{
Centre de Math\'ematiques Laurent Schwartz (CMLS), CNRS, \'Ecole polytechnique,
Institut Polytechnique de Paris, F-91120 Palaiseau, France}
\email{nicolas.perrin.cmls@polytechnique.edu}

\begin{abstract} 
We compute the small quantum cohomology  of Gushel-Mukai fourfolds. Following \cite{kkpy}, our computations imply that the very general ones are not rational.  Following \cite{jg}, and thanks to a suitable deformation of the small quantum cohomology ring, we also deduce that a rational Gushel-Mukai fourfold has the same rational cohomology as some K3 surface.
\end{abstract} 

\maketitle

\section{Introduction}

After the spectacular, long awaited proof that the very general complex cubic fourfold is not rational \cite{kkpy}, the theory of atoms 
introduced in this paper will certainly lead to many new exciting developments.

\smallskip
One of the first cases to which one would be tempted to apply the same approach is that of Gushel-Mukai fourfolds. Indeed, these well-studied Fano manifolds share many of the peculiar features of cubic fourfolds. 
In particular, their middle cohomology is of K3 type and allows one to define a period map, with lattice theoretic, sometimes geometric connections with some polarised K3 surfaces (and also cubic fourfolds, see \cite{dim}).
They are also related with families of hyperK\"ahler manifolds, double EPW sextics \cite{im}, exactly as cubic fourfolds to their Fano varieties of lines. Their derived category contains a K3 category, the so-called Kuznetsov component, whose behavior is also expected to be related to rationality \cite{kp}. 

\smallskip 
One difference with cubic fourfolds is that the latter have index $3$, while Gushel-Mukai fourfolds have index $2$. This makes the quantum cohomology a bit more complicated to compute. In \cite{kkpy}, the relevant information for cubic fourfolds is deduced from the work of Givental, as a special case of his computations for complete intersections in toric varieties. 
Gushel-Mukai varieties being outside the toric world, we compute directly its quantum cohomology -- more precisely, the small quantum products of ambient cohomology classes, those that are restricted from the Grassmannian they live in. This can readily be done using our understanding of lines and conics on Gushel-Mukai manifolds, and classical intersection theory. 

\begin{theorem}\label{matrix}
The quantum multiplication by the hyperplane class is given by the following matrix:
$$\begin{pmatrix}
 0&6q&0&0&24q^2&0 \\
 1&0&10q&6q&0&24q^2 \\
 0&1&0&0&4q&0 \\
 0&1&0&0&2q&0 \\
 0&0&3&2&0&6q \\
 0&0&0&0&1&0 
 \end{pmatrix}$$
 For $q$ generic, it has four nonzero simple eigenvalues and a two dimensional kernel.
 \end{theorem}
 
As a result, the quantum cohomology of a very general Gushel-Mukai contains several atoms that have the necessary features to be obstructions to rationality. 
To be more precise, the irrationality criterion of 
\cite[Proposition 5.30]{kkpy} allows one to conclude
(Theorem \ref{thm_irrationality_GM}) that:

\begin{quote}
\emph{A Hodge general Gushel-Mukai fourfold is irrational.}
\end{quote} 

We also compute the full multiplication table of quantum products of ambient classes, and deduce a presentation of this algebra (Theorem \ref{presentation}). Then we extend the results of \cite{jg}, who proved that the primitive cohomology of a rational cubic fourfold must be isomorphic to  
the degree two cohomology of some K3 surface, at least over $\QQ$. For Gushel-Mukai fourfolds the same approach does not extend on the nose, because the small quantum product by the  hyperplane class has a two dimension kernel in the space of ambient classes, rather than a non trivial Jordan block of size two
for cubic fourfolds. In order to overcome this problem, using our explicit computations, we are able to find a suitable deformation of the small quantum cohomology ring in the big one, for which the required Jordan
block does show up. We deduce: 

\begin{quote}
\emph{The rational primitive cohomology of a rational Gushel-Mukai fourfold must be Hodge isomorphic to the degree two rational cohomology of some K3 surface. }
\end{quote} 

\medskip\noindent {\it Acknowledgements}. We warmly thank Jérémy Guéré for sharing \cite{jg} and for many useful discussions. Thanks also to Franco Rota and Franco Giovenzana for the workshop they organized on \cite{kkpy}. All authors  were supported by the project FanoHK ANR-20-CE40-0023.

\section{Gushel-Mukai fourfolds} 

\subsection{Lines and conics on Gushel-Mukai fourfolds}
Consider a smooth (ordinary, in the terminology of \cite{debkuz}) Gushel-Mukai fourfold 
$$X=G(2,5)\cap \Omega_1\cap \Omega_2,$$  
where $\Omega_1$ and $\Omega_2$ are a hyperplane and a quadric in $\PP^9=\PP(\wedge^2\CC^5)$. We frequently use the notation $G(2,V_5)\subset \PP(\wedge^2V_5)$ for such a Grassmannian, $V_5$ being a $5$-dimensional vector space. Consistently, we usually denote a vector space by a capital letter indexed by the  dimension. 

We will always suppose in the sequel that $\Omega_1$ and $\Omega_2$ are generic. We fix equations that we denote by $\omega_1$ and $\omega_2$, respectively: $\omega_1$
corresponds to a skew-symmetric bilinear form on $V_5$, that we will denote in the same way; $\omega_2$ is a quadratic form on  $\wedge^2V_5$, and we will keep the same notation for the corresponding symmetric bilinear form. 

\medskip
Gushel-Mukai fourfolds are prime Fano varieties of index $2$. They contain a three-dimensional family of lines, and a five dimension family of conics. We record the following statement (see \cite[Proposition 5.3]{debkuz}, \cite[Theorem 3.2]{im}, \cite[Proposition 7.12]{debkuz2}):

\begin{prop}\label{smooth}
On a general Gushel-Mukai fourfold, the Hilbert schemes of lines and conics are smooth. 
\end{prop}

For a cubic fourfold, the Fano variety of lines has a famous hyperK\"ahler structure; 
for a Gushel-Mukai fourfold, the scheme of conics contracts, through a $\PP^1$-fibration, to another type of polarized hyperK\"ahler fourfolds, called double EPW sextics (see \cite[Proposition 4.18]{im}, \cite[Proposition 7.12]{debkuz2}, and also \cite{glz} for an interpretation in terms of Bridgeland moduli spaces). 

\subsection{Ambient cohomology}
A Gushel-Mukai fourfold has Hodge diamond 

\small
$$\begin{matrix}
&&&&1&&&&\\
&&&0&&0&&&\\
&&0&&1&&0&& \\
&0&&0&&0&&0&\\
0&&1&&22&&1&&0 \\
&0&&0&&0&&0&\\
&&0&&1&&0&& \\
&&&0&&0&&&\\
&&&&1&&&&
\end{matrix}$$
\normalsize

\medskip
Its {\it ambient cohomology} is the part of its rational cohomology ring inherited from $G(2,5)$. It  has dimension $6$, and is generated by $\sigma_0=1$, $\sigma_1$, $\sigma_2$, $\sigma_{11}$, $\sigma_3$, $\sigma_{31}$, the restrictions of the corresponding Schubert cycles on $G(2,5)$, that we denote in the same way.  We denote by $\sigma_0^\vee$, $\sigma_1^\vee$, $\sigma_2^\vee$, $\sigma_{11}^\vee$, $\sigma_3^\vee$, $\sigma_{31}^\vee$ the Poincaré dual basis. For any two classes $ \sigma$ and  $\sigma'$, we have 
\begin{equation}\label{restriction}
\int_X \sigma\sigma' = 2\int_{G(2,5)}\sigma\sigma'\sigma_1^2.
\end{equation}
We readily deduce that
$$\sigma_0^\vee =\pt = \frac{1}{2}\sigma_{31}=\frac{1}{2}\sigma_{22},\qquad
\sigma_1^\vee = \frac{1}{2}\sigma_3=
\frac{1}{4}\sigma_{21},$$
$$\sigma_2^\vee= \frac{1}{2}(\sigma_2-\sigma_{11}),\qquad 
\sigma_{11}^\vee = \sigma_{11}-\frac{1}{2}\sigma_2,$$  $$\sigma_3^\vee=\frac{1}{2}\sigma_1,\qquad  \sigma_{31}^\vee = \frac{1}{2}\sigma_0.$$

\medskip
Here we also denoted by $\pt$ the point class of $X$. In our formulas we often omit $\sigma_0$.

\medskip
For $X$ very general, every Hodge class is inherited from $G(2,5)$ by \cite[Corollary 4.6]{dim}:
$$H^\bullet(X)^{\mathrm{Hod}}=H^\bullet(X)^{\mathrm{amb}}=\QQ\langle \sigma_0, \sigma_1, \sigma_2, \sigma_{11}, \sigma_3, \sigma_{31}\rangle .$$
Moreover $H^\bullet(X)^{\mathrm{amb}}$ is a subalgebra of
$H^\bullet(X)$, for the usual cup-product which is easily determined. 
The multiplication by the hyperplane class $\sigma_1$ is obtained by restricting the usual Pieri rule on $G(2,5)$, which yields 
$$\sigma_0\mapsto\sigma_1, \quad \sigma_1\mapsto\sigma_2+\sigma_{11}, \quad \sigma_2\mapsto 3\sigma_{3}, \quad \sigma_{11}\mapsto 2\sigma_{3}, \quad \sigma_{3}\mapsto \sigma_{31},\quad \sigma_{31}\mapsto 0.$$
Only one extra computation is needed to determine the usual product on $H^\bullet(X)^{\mathrm{amb}}$ completely; from formula (\ref{restriction}) we deduce that $$\sigma_2^2=4 \pt, \quad  \sigma_2\sigma_{11}=\sigma_{11}^2=2 \pt .$$

\section{Quantum multiplication by the hyperplane class}
The ambient cohomology $H^\bullet(X)^{\mathrm{amb}}$ is also a subalgebra for the quantum product; in particular it is preserved by quantum multiplication with the hyperplane class. In the next subsection, we will compute the full quantum product of ambient classes.

Since $X$ has index $2$ and dimension four, we will only need Gromov-Witten invariants of degree one and two to determine the quantum multiplication by the hyperplane class. 
If $\alpha_1,\dots,\alpha_n$ are cohomology classes, we denote by $\langle \alpha_1,\dots,\alpha_n\rangle_d$ the degree $d$, $n$-pointed Gromov-Witten invariant with insertions $\alpha_1,\dots,\alpha_n$. This invariant is symmetric with respect to $\alpha_1,\dots,\alpha_n$. Moreover, 
 if $\alpha_1=\sigma_1$ is the hyperplane class, then 
$\langle \alpha_1,\dots,\alpha_n\rangle_d$=$d\langle \alpha_2,\dots,\alpha_n\rangle_d$ by the divisor axiom. 
The quantum multiplication by the hyperplane class is a priori of the form
$$\sigma_1\star \sigma_0 = \sigma_1, \qquad \sigma_1\star \sigma_1 = \sigma_2+\sigma_{11}+I^{(1)}_1q\sigma_0, $$
$$\sigma_1\star \sigma_2 = 3\sigma_{3}+I_2^{(1)}q\sigma_1, \qquad \sigma_1\star \sigma_{11} = 2\sigma_{3}+I_3^{(1)}q\sigma_1,$$
$$\sigma_1\star \sigma_3 = \sigma_{31}+2I_2^{(1)}q\sigma_{2}^\vee+2I_3^{(1)}q\sigma_{11}^\vee+2I^{(2)}q^2\sigma_0,$$
$$\sigma_{31}\star \sigma_1 = I^{(1)}_1q\sigma_{3}+2I^{(2)}q^2\sigma_1,$$

\medskip\noindent
in terms of the following two-pointed Gromov-Witten 
invariants:
$$I^{(1)}_1:=\langle \sigma_1, \sigma_0^\vee\rangle _1=\frac{1}{2}\langle  \sigma_1, \sigma_{31}\rangle _1=\langle \sigma_{31}, \sigma_{3}^\vee\rangle _1,$$
$$I^{(1)}_2=\langle \sigma_2, \sigma_1^\vee\rangle _1=\frac{1}{2}\langle \sigma_2, \sigma_{3}\rangle _1, \qquad 
I^{(1)}_3=\langle \sigma_{11}, \sigma_1^\vee\rangle _1=\frac{1}{2}\langle \sigma_{11}, \sigma_{3}\rangle _1, $$
$$I^{(2)}=\langle \sigma_{3},  \sigma_{0}^\vee\rangle _2=\frac{1}{2}\langle \sigma_{31}, \sigma_3\rangle _2=\langle \sigma_{31}, \sigma_{1}^\vee\rangle _2.$$

\medskip
So we have only four Gromov-Witten invariants to compute, and we will compute them geometrically. As a matter of fact, these invariants are enumerative because, as we recalled in Proposition \ref{smooth}, the Hilbert schemes of lines and conics of $X$ are smooth, hence their virtual fundamental classes are just the usual fundamental classes; moreover the classes we evaluate are restricted from the Grassmannian, hence can be moved freely. Therefore, by choosing general representatives of these classes, their pull-back via the evaluation morphisms to the moduli space of pointed maps can (and will) be supposed to intersect transversely among themselves and with the virtual fundamental class; hence our claim that the Gromov-Witten invariants we need to compute are all enumerative.

\subsection*{Computation of $I^{(1)}_1$.} This is the number of lines in $X$ through a generic point defined by a plane $A_2$. 
Lines in $G(2,5)$ through this point are parame\-tri\-zed by $\PP(A_2)\times \PP(V_5/A_2)$, or 
more explicitly by pairs $(B_1,B_3)$ such that $B_1\subset A_2\subset B_3$. They are contained in $X$ 
if $\omega_1(B_1,B_3)=0$ and $\omega_2(B_1\wedge B_3)=0$. Since we already know that $[\wedge^2A_2]$ 
is a point of $X$, 
this reduces to the conditions $\omega_1(B_1,B_3/A_2)=0$ and 
$\omega_2(\wedge^2A_2,B_1\otimes (B_3/A_2))=\omega_2(B_1\otimes (B_3/A_2))=0$.
If $h$ and $H$ denote the hyperplane classes on    $\PP(A_2)$ and $ \PP(V_5/A_2)$, 
respectively, we get the intersection in $\PP(A_2)\times \PP(V_5/A_2)$ of three 
hypersurfaces of 
classes $h+H, h+H$ and $2(h+H)$. So  $$I^{(1)}_1=\int_{\PP^1\times \PP^2}2(h+H)^3=6.$$

\noindent  {\it Remark.}
For another computation of this invariant, see \cite[Lemma 5.6]{debkuz}.

\subsection*{Computation of $I^{(1)}_3$.}

Since $I^{(1)}_3=\langle \sigma_{11}, \mathrm{line}\rangle _1$, we consider a generic line $\Delta$ in $X$ defined by a pair $(B_1,B_3)$, and a generic Schubert cycle of class $\sigma_{11}$ -- this is just the intersection of $X$ with a subGrassmannian $G(2,V_4)$, hence a del Pezzo surface $\Sigma$ of degree four. If $U_2$ (resp. $U'_2$) is a plane parametrized by a point of $\Delta$ (resp. $\Sigma$), the line joining these points will be contained in $G(2,5)$ when $U_2$ and $U'_2$ meet non-trivially, and then contained in $X$ when $\omega_2(\wedge^2U_2,\wedge^2U'_2)=0$. The first condition can be expressed by the fact that we lie in the degeneracy locus of the natural morphism $U_2\oplus U'_2\lra V_5$. According to the Thom-Porteous formula, the class of this locus is $s_2(U_2\oplus U'_2)$, where $s_2$ denotes the second Segre class and we consider $U_2$ and $U'_2$ as vector bundles on $\Delta$ and $\Sigma$, respectively. We deduce that 
$$I^{(1)}_3=\int_{\Delta\times \Sigma}s_2(U_2\oplus U'_2)(c_1(U_2)+c_1(U'_2))=\int_\Sigma(s_2(U'_2)+c_1(U'_2)^2)=6.$$

\subsection*{Computation of $I^{(1)}_2$.}

Since $I^{(1)}_2=\langle \sigma_{2}, \mathrm{line}\rangle _1$, as before we consider a generic line $D$ in $X$ defined by a pair $(B_1,B_3)$, and a generic Schubert cycle of class $\sigma_{2}$, which is the set of planes intersecting non trivially a fixed plane $V_2$. This Schubert cycle is singular at $[V_2]$
(which is not a point of $X$ is general), and it is convenient to replace it by its obvious desingularization $\Gamma$, parametrizing pairs $(U_1\subset U_2)$ with $U_1\subset V_2$ -- a $\PP^3$-bundle above $\PP(V_2)$. Its intersection $\Sigma$ with $X$ is obtained by imposing that 
$\omega_1$ and $\omega_2$ vanish on $U_1\wedge U_2$. By the same argument as in the previous case, we deduce that 
$$I^{(1)}_2=\int_\Sigma(s_2(U_2)+c_1(U_2)^2)=2\int_\Gamma c_1(U_1\wedge U_2)^2(s_2(U_2)+c_1(U_2)^2).$$
Let $h=-c_1(U_1)$ be the hyperplane class of $\PP(V_2)$ and $H=-c_1(U_2/U_1)$. Then $c_1(U_1\wedge U_2)=c_1(U_2)=-h-H$, while $s_2(U_2)=H^2+hH$. Since  $\int hH^3=1$ and $\int H^4=-c_1(V_4/U_1)=-1$, we 
finally get 
$$I^{(1)}_2=2\int_\Gamma (H+h)^3(2H+h) =2\int_\Gamma (2H^4+7H^3h)= 2(-2+7)=10.$$

\subsection*{Computation of $I^{(2)}$.}  Since $\sigma_{31}$ is twice the punctual class of $X$, the Gromov-Witten invariant  
$I^{(2)}=2\langle \pt, \sigma_{1}^\vee\rangle _2$
where $ \sigma_{1}^\vee$ is the class of a line. So 
$I^{(2)}$ equals twice the number of conics passing through a generic point $A_2$ of $X$ and meeting a generic line $D$ in $X$, defined by a pair $(B_1,B_3)$.

Recall that a conic in $G(2,V_5)$ that is not a double line must be a linear section of some subGrassmannian $G(2,V_4)$. Hence the natural parameter space to consider is the Grassmann bundle $G(3,\wedge^2V_4)$ over $\PP(V_5)^\vee$. 
If we incorporate the conditions that the conic passes through $\wedge^2A_2$ (which implies that $A_2\subset V_4$), and that it is contained in $\Omega_1$, we are reduced to the Grassmann bundle 
$G(2,(\wedge^2V_4\cap \overline{\Omega}_1)/\wedge^2A_2)$ over $\PP(V_5/A_2)^\vee$, where we denote by  $\overline{\Omega}_1$ the hyperplane of $\wedge^2V_5$ corresponding to $\Omega_1$.
Note that since $\overline{\Omega}_1$ is generic, it never contains  $\wedge^2V_4$ ; in particular we get a honest $G(2,4)$-bundle over $\PP^2$. 

Consider a three-dimensional space $L_3$, with $\wedge^2A_2\subset L_3\subset \wedge^2V_4\cap \overline{\Omega}_1$. If it meets the line $D$, this will necessarily be at the point defined by $B_2=B_3\cap V_4$; in particular $V_4$ has to contain $B_1$, and $L_3$ has to contain $\wedge^2 B_3\cap V_4$. This reduces again our parameter space to a $\PP^2$-bundle $\Theta$ over $\PP(V_5/(A_2\oplus B_1))\simeq \PP^1$.

Now we need to impose the condition that the intersection of $G(2,V_4)$ with $\PP(L_3)$ is contained in $X$, or equivalently in $\Omega_2$. This means that the quadratic form on $L_3$ defined by $\omega_2$, and the Pl\"ucker quadratic form
$S^2L_3\hookrightarrow S^2(\wedge^2V_4)\lra \wedge^4V_4$, must be proportional. We already know that these two quadratic forms vanish on $a=\wedge^2A_2$ and $b=\wedge^2(B_3\cap V_4)$, so we only need to check they are proportional in the space $M_4\subset S^2L_3^\vee$ of quadratic forms vanishing on those lines. Finally, our condition is to belong to the degeneracy locus of the corresponding morphism $\cO\oplus \det(V_4)^\vee\lra M_4$. Since $M_4$ is a rank four vector bundle, the expected codimension is correct and we can apply the Thom-Porteous formula, to deduce that 
$$\frac{1}{2}I^{(2)}=\int_\Theta c_3(M_4-\det(V_4)^\vee).$$
Let $h=-c_1(V_4)$ be the hyperplane class. In order to compute this intersection number, we can use the exact sequence $$0\lra M_4\lra S^2L_3^\vee\lra a^{-2}\oplus b^{-2}=\cO\oplus\cO(2h)\lra 0.$$
This allows us to express $I^{(2)}$ as 
$$\frac{1}{2}I^{(2)}=\int_\Theta \frac{c(M_4)}{1+h}=\int_{\Theta}\frac{c(S^2L_3^\vee)}{(1+h)(1+2h)}.$$
On the other hand, recall that $L_3$ must contain $L_2=a+b$, hence 
$$c(S^2L_3^\vee)=c(S^2L_2^\vee)c(L_2^\vee\otimes \ell)c(\ell^2),$$
where $\ell=c_1((L_3/L_2)^\vee)$ is a relative hyperplane class. Since the vector bundle $L_2\simeq\cO\oplus\cO(-h)$, we  have $c_1(L_2^\vee)=h$ and $c_1(S^2L_2^\vee)=3h$. 
We deduce that

$$\frac{1}{2}I^{(2)}=\int_\Theta (1+\ell)(1+\ell+h)(1+2\ell)=
2\int_\Theta\ell^3+2\int_\Theta h\ell^2.$$

\medskip\noindent
But  $\int h\ell^2=1$ and $\int \ell^3=-\int_{\PP^1}c_1((\wedge^2V_4\cap H)/L_2)=2$, so finally $I^{(2)}=12$.  \medskip 

\subsection*{Quantum products with primitive classes}
The following result is not strictly needed, and could also be deduced from the  irrationality arguments that appear in the next section. However we can give here an independent (quite standard) proof, that we include for the sake of completeness. 

Recall that if $j:X \hookrightarrow G(2,5)$ denotes the natural embedding, the primitive
cohomology of $X$ is $H^\bullet(X)^{\mathrm{prim}}:=\ker (h\cup(\cdot)) $; it coincides with the vanishing cohomology $H^\bullet(X)^{\mathrm{van}}:=\ker(j_*)$ and also by the projection formula with 
the orthogonal of the ambient cohomology   $H^\bullet(X)^{\mathrm{amb}}$ with respect to the intersection pairing. The following statement is similar to  \cite[Lemma 6.11]{kkpy}.

\begin{prop}
\label{prop_prim_action}
For any $\gamma\in H^\bullet(X)^{\mathrm{prim}}$,  we have $h\star \gamma=0$.
\end{prop}

\begin{proof}
It is well-known that Gromov-Witten invariants are monodromy-invariant. Moreover, the monodromy acts irreducibly on the primitive cohomology of hypersurfaces \cite[Theorem 3.27]{Vo}. For $\alpha,\beta\in H^\bullet(X)^{\mathrm{amb}}$, $d\in \mathbb{N}$, let us define  $$\psi_{\alpha,\beta}^{(d)}=\langle \alpha,\beta,\bullet \rangle_d : H^\bullet(X)^{\mathrm{prim}} \to \CC. $$
This is a monodromy invariant element inside the irreducible representation $(H^\bullet(X)^{\mathrm{prim}})^\vee$, so it has to vanish. 

Now, the product $h\star \gamma$ has degree three, while the degrees of $\gamma$ and $q$ are both even. Therefore, if $\langle h,\gamma,\beta\rangle_d$ is a Gromov-Witten invariant appearing in $h\star \gamma$, then $\deg(\beta^\vee)$ and $\deg(\beta)$ must be odd, which forces $\beta$ to be an ambient class. But then, by the above argument, $\langle h,\gamma,\beta\rangle_d=
\psi_{a,\beta}^{(d)}(\gamma)=0$.
\end{proof}

\section{Irrationality of the very general Gushel-Mukai fourfold}
The quantum multiplication by the hyperplane class is given, according to the previous computations, in the basis $\sigma_0=1$, $\sigma_1$, $\sigma_2$, $\sigma_{11}$, $\sigma_3$, $\sigma_{31}$, by the matrix
\begin{equation*}
\label{eq_mult_h}
    \begin{pmatrix}
 0&6q&0&0&24q^2&0 \\
 1&0&10q&6q&0&24q^2 \\
 0&1&0&0&4q&0 \\
 0&1&0&0&2q&0 \\
 0&0&3&2&0&6q \\
 0&0&0&0&1&0 
 \end{pmatrix}
\end{equation*}
 
 \medskip
 The characteristic polynomial is $X^2P(X^2)$, where
$P(T)=T^2-44qT-16q^2$ has roots $T_1q$ and $T_2q$ with 
$$T_1=22+10\sqrt{5}>0>T_2=22-10\sqrt{5}.$$
Applying the ideas of \cite{kkpy}, we can deduce:

\begin{theorem}
\label{thm_irrationality_GM}
The very general Gushel-Mukai fourfold is irrational.
\end{theorem}

\begin{proof}[Proof, following \cite{kkpy}]
We will apply the irrationality criterion of \cite[Proposition 5.30]{kkpy}. According to this criterion, if a Hodge atom appears in a rational fourfold $Z$,  then it has to appears either in $\PP^4$, a surface, a curve or a point. 

To a Hodge atom $\boldsymbol\alpha$ on a variety $Y$, and any generic $\xi \in H^\bullet(Y)^{\mathrm{Hdg}}$, one can associate an eigenvalue $\alpha$ of the quantum multiplication $\kappa_\xi$ by $\mathrm{Eu}_\xi$. Here, $\mathrm{Eu}_\xi$ is the Euler vector field in a deformation (in the big quantum cohomology of $Y$) of the small quantum cohomology along  $\xi$. Let $E_{\xi,\alpha}$ be the generalized eigenspace of $\kappa_\xi$ for this eigenvalue $\alpha$. One defines $$\rho_{\xi,\alpha}^Y:=\dim_\CC (E_{\xi,\alpha}\cap QH^\bullet(Y)^{\mathrm{Hdg}}) \quad\mathrm{and}\quad \nu_{\xi,\alpha}^Y:= \dim_\CC(E_{\xi,\alpha}\cap H^{(2)}(Y)),$$ where $H^{(2)}(Y)$ is the degree two Hochschild cohomology of $Y$. It is shown in the proof of \cite[Theorem 6.8]{kkpy} that, if $\boldsymbol\beta$ is a Hodge atom of a variety $Y$ which is either $\PP^n$, a surface, a curve or a point then either $\nu_{\xi,\beta}^Y=0$ or $\rho_{\xi,\beta}^Y\geq 3$. This implies that if a fourfold $Z$ is rational then, for any Hodge atom $\boldsymbol\alpha$ in $Z$, either $\nu_{\xi,\alpha}^Z=0$ or $\rho_{\xi,\alpha}^Z\geq 3$.

Now, the operator $\kappa_\xi$ is a deformation of $\kappa_0$. So the decomposition in generalized eigenspaces of $\kappa_\xi$ is finer than that of $\kappa_0$, in the sense that it specializes at  $\xi_0=0$  to a decomposition in which each factor is contained in a (unique) generalized eigenspace for $\kappa_0$. For $\xi=0$ we are just considering the small quantum cohomology, and the operator $\kappa_0$ is the (small quantum) multiplication by $c_1(Z)$. This leads to the following irrationality criterion
(which can be led back directly to the proof of \cite[Theorem 6.8]{kkpy}, see \cite[Definition 27]{jg} and \cite[Corollary 46]{jg}): 
\begin{quote}$(\clubsuit)$
{\it If $Z$ is a Fano fourfold with $H^{3,1}(Z)\neq 0$, such that the multiplicity of any eigenvalue for the action of $\kappa_0=c_1(Z)\star$ on $H^\bullet(X)^{\mathrm{Hdg}}$ is $\leq 2$, then $Z$ is irrational.}
\end{quote}
Indeed, if this is the case, any Hodge atom $\boldsymbol\alpha$ of $Z$ will satisfy $\rho_{\xi,\alpha}^Z\leq 2$ by the previous deformation argument. Moreover, since $H^{3,1}(Z)\neq 0$, there must exist a Hodge atom $\boldsymbol\alpha_0$ satisfying $\nu_{\xi,\alpha_0}^Z\neq 0$. However, such an atom $\alpha_0$ cannot appear as a Hodge atom of $\PP^4$, a surface, a curve or a point. 

\medskip
Consider now a (very general) Gushel-Mukai fourfold $X$, for which the Hodge classes coincide with the ambient classes (see \cite[Corollary 4.6]{dim}). The Euler vector field $Eu_0$ is just  $c_1(X)=2h$, so $\kappa_0=2h\star(\cdot)$ is twice the product by the hyperplane class in the small quantum cohomology ring. From Theorem \ref{matrix} and Proposition \ref{prop_prim_action}, we know that $\kappa_0$ has five generalized eigenspaces: four of dimension one, and the kernel $E_0$  whose dimension is $24$. Moreover, $E_0$ contains the nonzero subspace $H^{3,1}(X)\subset H^\bullet(X)^{\mathrm{prim}}$, of Hochschild degree two, and $$\dim_\CC (E_0\cap QH^\bullet(X)^{\mathrm{amb}})=2.$$
We can thus deduce from $(\clubsuit)$ that $X$ is irrational.
\end{proof}

\medskip\noindent {\it Remark.}  We can give a more compact proof by following more closely 
the ideas of  
\cite{jg}, whose notations we also follow (see \cite[Section 2]{jg}). 
In particular we consider the non-Archimidean field $F=\QQ((a^\QQ))$ and denote $S^*=F[b^{\pm 1}]$.  Consider again a Hodge general Gushel-Mukai fourfold $X$. Denote by $t_0,\dots,t_5$ the dual basis to the basis $\sigma_0=1$, $\sigma_1$, $\sigma_2$, $\sigma_{11}$, $\sigma_3$, $\sigma_{31}$ of $H^\bullet(X)^{\mathrm{amb}}$, and let $\hat{R}^*(X,\QQ)=\QQ[[q,t_0,\dots,t_5]]$.   
Let $\zeta\in D(0,1)$, the open unit disk in $F$ with respect to its absolute value function. Consider the $\QQ$-evaluation function (\cite[Definition 20]{jg})
$$
\mathrm{ev}:\left\{q,t_0,\dots t_5 \right\} \to S^*, 
$$
$$
q\mapsto \zeta^2 b^4, \quad t_k\mapsto 0 \quad \forall k.
$$
By \cite[Proposition 21]{jg}, this defines a $\QQ$-evaluation map 
$$\mathrm{ev}:\hat{R}^*(X,\QQ)=\QQ[[q,t_0,\dots,t_5]]\to S^*$$ 
 For $\tau=t_0\sigma_0+t_1\sigma_1+t_2\sigma_2+t_3\sigma_{11}+t_4\sigma_3+t_5\sigma_{31}$, let again $\kappa_\tau$ denote the multiplication operator by the Euler vector field $\mathrm{Eu}_\tau$ in the big quantum cohomology ring (\cite[Definitions 11 and 12]{jg}). 
 Then the evaluation $\mathrm{ev}(\kappa_\tau)$
is just the multiplication operator by $c_1(X)=2h$ in the small quantum cohomology ring. By the same arguments as in the previous proof of Theorem \ref{thm_irrationality_GM}, we deduce that property $\clubsuit_{\hat{R}^*(X,\tilde{\mathbb{Q}}_{\mathrm{ext}})}$ (\cite[Definition 27]{jg}) is not satisfied (see \cite{jg} for the definition of $\tilde{\mathbb{Q}}_{\mathrm{ext}}$), which implies by \cite[Remark 29]{jg} that property $\clubsuit_{\hat{R}^*(X,R)}$ is not satisfied either, for any number field $R$ containing $\tilde{\mathbb{Q}}_{\mathrm{ext}}$. Then by \cite[Corollary 46]{jg}, $X$ is irrational.

\medskip\noindent {\it Remark.} We have considered so far only ordinary Gushel-Mukai fourfolds, as opposed to special ones which are double covers of a codimension two linear section of $G(2,5)$, branched over a quadric
(those are called ‘‘of Gushel type" in \cite{im}).

Special Gushel-Mukai fourfolds can of course be deformed into ordinary ones, simply by considering the cone over the Grassmannian and sections by two hyperplanes and a quadric. Their ambient classes can be defined in the same way as those coming from $G(2,5)$, and the quantum products by the hyperplane class are given by the same formulas as in the ordinary case, since the Gromov-Witten invariants are deformation invariants. So the previous statement also holds for special Gushel-Mukai fourfolds. Note that by \cite[Corollary 4.6]{dim}, for a very general special Gushel-Mukai fourfold, it is also true that the only Hodge classes are the ambient ones. So the very general special Gushel-Mukai fourfold is irrational.

 \medskip It remains a very interesting question to characterize rational Gushel-Mukai fourfolds, either in terms of Hodge structure or derived category. 
 As explained in \cite[Remark 1.10]{ppz}, where the Hodge-theoretic version can be found, these two different points of view lead to two different conjectural conditions for rationality, the categorical version being \cite[Conjecture 3.12]{kp}. This is in contrast to the case of cubic fourfolds for which the corresponding conditions are known to be equivalent \cite{addington-thomas}. We will contribute to this question in the last section of this paper.

 \medskip\noindent {\it Remark.} Examples of rational Gushel-Mukai fourfolds are constructed in \cite{dim, hs, sta}.

\section{The ambient quantum cohomology}
\label{sec_ambient}
In this section we complete our computations of the quantum products of ambient classes. 
In order to get a presentation of the ambient quantum cohomology, we just need to compute the product $\sigma_2\star \sigma_2$, or equivalently $\sigma_{11}\star \sigma_2$ or $\sigma_{11}\star \sigma_{11}$. Indeed, since $h\star h=\sigma_2+\sigma_{11}+6q$, these three products can be deduced one from the other through products by the hyperplane class, that 
have been dealt with in the previous sections. 

Since $\sigma_{11}\cup \sigma_{11}=2\pt$, we can decompose 
$$
\sigma_{11}\star \sigma_{11} =2\pt + J^{(1)}_1 q\sigma_2^\vee +J^{(1)}_2 q\sigma_{11}^\vee + 
J^{(2)} q^2,
$$ 
in terms of the following three-pointed Gromov-Witten invariants: 
$$J^{(1)}_1=\langle \sigma_{11},\sigma_{11},\sigma_{2} \rangle_1, \quad J^{(1)}_2=\langle \sigma_{11},\sigma_{11},\sigma_{11} \rangle_1, \quad  J^{(2)}=\langle \sigma_{11},\sigma_{11},\pt \rangle_2.$$

\medskip\noindent
From the formulas 
$$\sigma_2\star h\star h = 6\pt+12I^{(1)}_2q\sigma_2^\vee+(4I^{(1)}_2+6I^{(1)}_3)q\sigma_{11}^\vee+(3J^{(2)}+I^{(1)}_1I^{(1)}_2)q^2, $$
$$\sigma_{11}\star h\star h = 4\pt+(4I^{(1)}_2+6I^{(1)}_3)q\sigma_2^\vee+8I^{(1)}_3 q\sigma_{11}^\vee+(2I^{(2)}+I^{(1)}_1I^{(1)}_3)q^2, $$

\medskip\noindent
we deduce that the other quantum products of degree two classes are
$$\sigma_{11}\star \sigma_2 =2\pt +(4I^{(1)}_2+6I^{(1)}_3-J^{(1)}_1-2I^{(1)}_1)q\sigma_2^\vee+\hspace*{3cm}$$
$$\hspace*{15mm}+(8I^{(1)}_3-J^{(1)}_2-2I^{(1)}_1)q\sigma_{11}^\vee+(2I^{(2)}-J^{(2)}+I^{(1)}_1I^{(1)}_3)q^2,$$
$$\sigma_2\star \sigma_2 =4\pt +(8I^{(1)}_2-6I^{(1)}_3+J^{(1)}_1-2I^{(1)}_1)q\sigma_2^\vee+\hspace*{3cm}$$
$$\hspace*{22mm}+(4I^{(1)}_2-2I^{(1)}_3+J^{(1)}_2)q\sigma_{11}^\vee+(I^{(2)}+J^{(2)}-I^{(1)}_1(I^{(1)}_2+I^{(1)}_3))q^2.$$

\medskip
As a result, we get two different expressions for the following Gromov-Witten invariants: 
$$\langle \sigma_{11},\sigma_{11},\sigma_{2}\rangle_1=J^{(1)}_1=8I^{(1)}_3-J^{(1)}_2-2I^{(1)}_1, $$
$$\langle \sigma_{11},\sigma_2,\sigma_{2}\rangle_1=
4I^{(1)}_2+6I^{(1)}_3-J^{(1)}_1-2I^{(1)}_1= 4I^{(1)}_2-2I^{(1)}_3+J^{(1)}_2.$$

\medskip
Both equalities imply the same relation, namely:

\begin{lemma}
$J^{(1)}_1+J^{(1)}_2=8I^{(1)}_3-2I^{(1)}_1=36.$
\end{lemma}

Our next task will be to compute $J^{(1)}_2=\langle \sigma_{11},\sigma_{11},\sigma_{11} \rangle_1$ geometrically. We will immediately deduce $J^{(1)}_1$.

\subsection*{Computation of $J^{(1)}_2$.} A general subvariety representing the class $\sigma_{11}$ is the set of planes in $X$ contained in a given four-dimensional subspace $A_4\subset V_5$. We choose three such subvarieties in general position, defined by $A_4,B_4,C_4\subset V_5$, and we look for the number of  lines in $X$ passing through these three subvarieties. 

A line in $G(2,V_5)$ is defined by a flag $L_1\subset L_3\subset V_5$. That it meets our three subvarieties means that there exist planes $A_2\subset A_4, B_2\subset B_4, C_2\subset C_4$ such that $L_1\subset A_2, B_2, C_2\subset L_3$. In particular $L_1\subset W_2:=A_4\cap B_4\cap C_4$ and $A_2\subset L_3\cap A_4$, $B_2\subset L_3\cap B_4$,  $C_2\subset L_3\cap C_4$; moreover, for a general line these three inclusions will be equalities. 

We can then use for parameter space $\Sigma$ the space of pairs $(L_1\subset L_3)$ such that 
$L_1\subset W_2$;  
this is a $G(2,4)$-bundle over $\PP(W_2)$, more precisely, the bundle  $G(2,V_5/\cO(-1))$ over $\PP^1\simeq \PP(W_2)$. 
We need to impose that the line defined by the pair $(L_1\subset L_3)$ is contained in  $X$, that is, that $\omega_1$ and $\omega_2$ vanish on $L_1\wedge L_3\subset \wedge^2V_5$. 
This gives $2+3=5$ conditions on the five-dimensional parameter space $\Sigma$, yielding as expected a finite number of lines. For $\omega_1$ and $\omega_2$ generic, this number can be computed as 
$$J^{(1)}_2=\int_\Sigma c_2((L_1\wedge L_3)^\vee)c_3(S^2(L_1\wedge L_3)^\vee).$$
Let $h=-c_1(L_1)$ and $1+H+\alpha=c((L_3/L_1)^\vee)$, where $\alpha$ is the second Chern class. Since $L_1\wedge L_3\simeq L_1\otimes (L_3/L_1)$ as vector bundles on $\Sigma$, we can compute 
$c_2((L_1\wedge L_3)^\vee)= \alpha+hH$ and
$c_3(S^2(L_1\wedge L_3)^\vee)=8hH^2+16\alpha h+4\alpha H$. Hence
$$
J^{(1)}_2=\int_\Sigma (\alpha+hH)(4hH^2+8\alpha h+4\alpha H)= \int_\Sigma h(8H^2\alpha+8\alpha^2) +\int_\Sigma 4\alpha^2H=
$$
$$
=\int_{G(2,4)}(8H^2\alpha+8\alpha^2) + \int_\Sigma 4\alpha^2H = 8+8+4 \int_\Sigma \alpha^2H.
$$
Now, $\Sigma=G(2,V_5/\cO(-1))$, and $V_5/\cO(-1)$ is an extension of $\cO(1)$ with $\cO^3$. So the cohomology ring of $\Sigma$ is $$H^*(\Sigma,\QQ)=\QQ[h,H,\alpha,H',\alpha']/(h^2,1-h-(1+\alpha+H)(1+H'+\alpha')),$$
where $c((V_5/L_3)^\vee)=1+H'+\alpha'$ and $\alpha'$ is the second Chern class. 
The relation $1-h=(1+H+\alpha)(1+H'+\alpha')$
can be deduced from the identity of Chern classes
$$
c((V_5/\cO(-1))^\vee)=c((V_5/L_3))^\vee) c((L_3/L_1)^\vee),
$$
which follows 
from the dual tautological exact sequence $$0\to  (V_5/L_3)^\vee \to (V_5/\cO(-1))^\vee \to (L_3/L_1)^\vee \to 0.$$
We deduce that $\alpha^2H=-\alpha^2h=-1$, and therefore $J^{(1)}_2=16-4=12$ and $J^{(1)}_1=24$.

\subsection*{Computation of $J^{(2)}$}
For this last computation, we will use once again the symmetry properties of the Gromov-Witten invariants. Indeed, we can compute in terms of $J^{(2)}$ the quantum products of any two ambient classes. We find in particular that 
$$\pt\star \sigma_2=14q^2\sigma_2+12q^2\sigma_{11}+(120-J^{(2)})q^3, $$
$$\pt\star \sigma_{11}=10q^2\sigma_2+6q^2\sigma_{11}+(24+J^{(2)})q^3. $$
As a consequence, the Gromov-Witten invariant
$$\langle \pt, \sigma_{11},\sigma_{11}\rangle_2=32.$$
But this is $J^{(2)}$ by definition! 

\medskip\noindent {\it Remark.} If we don't plug in the values of the coefficients we computed before, the symmetries of the Gromov-Witten invariants yield the formula  
\begin{equation*}\label{rel2}
J^{(2)}=I^{(1)}_1(2I^{(1)}_3-I^{(1)}_2)+
\frac{1}{5}(2I^{(1)}_2-3I^{(1)}_3)(2I^{(1)}_2+9I^{(1)}_3-\frac{5}{2}J^{(1)}_1)+\frac{3}{5}I^{(2)}.
\end{equation*}

\subsection*{Conclusion}
We can now easily complete the multiplication table:
$$\begin{array}{rcl}
\sigma_2\star\sigma_2 & = & 2\sigma_{31}+8q\sigma_2+12q\sigma_{11}+80q^2, \\
\sigma_2\star\sigma_{11} & = & \sigma_{31}+8q\sigma_{2}+4q\sigma_{11}+52q^2, \\
\sigma_{11}\star \sigma_{11}& = & \sigma_{31}+6q\sigma_2+32q^2, \\
\sigma_3\star\sigma_2 & = & 10q\sigma_3+60q^2\sigma_1, \\
\sigma_3\star\sigma_{11} & = & 6q\sigma_3+40q^2\sigma_1, \\
\sigma_3\star\sigma_{3} & = & 20q^2\sigma_2+20q^2\sigma_{11}+120q^2, \\
\sigma_{31}\star\sigma_2 & = & 28q^2\sigma_2+24q^2\sigma_{11}+176q^3, \\
\sigma_{31}\star\sigma_{11} & = & 20q^2\sigma_2+12q^2\sigma_{11}+112q^3, \\
\sigma_{31}\star\sigma_3 & = & 24q^2\sigma_3+120q^3\sigma_1, \\
\sigma_{31}\star\sigma_{31} & = & 64q^3\sigma_2+48q^3\sigma_{11}+368q^4. 
\end{array}$$ 

\medskip
We   deduce the following relations:
$$
R_1:=5h\sigma_{11}-2h^3+14qh,
$$
$$
R_2:=5\sigma_{11}^2+20q\sigma_{11}-h^4+12qh^2+20q^2,
$$
$$
R_3:=h^5-44qh^3-16q^2h.
$$

\medskip
This yields a presentation of the quantum ambient cohomology ring of a Gushel-Mukai fourfold:

\begin{theorem}\label{presentation}
$QH^\bullet(X)^{\mathrm{amb}}=\QQ[h,\sigma_{11},q]/(R_1,R_2,R_3)$.
\end{theorem}

\begin{proof}
The relations directly come from the quantum products we have computed so far, expressed in terms of $h,\sigma_{11},q$. We claim they generate the ideal of relations. 
Indeed, $R_1$ and $R_2$ express $h\star \sigma_{11}$ and $\sigma_{11}\star\sigma_{11}$ in terms of $\sigma_{11}$, $h,h^2,h^3,h^4$ and $q$. There is no relation between $\sigma_{11},h,h^2,h^3,h^4$ because they are $\CC$-linearly independent in the classical cohomology. They also $\CC$-linearly generate the classical cohomology, so $h^5$ has to depend on these, and this dependence is expressed in $R_3$.
\end{proof}

\section{On rational Gushel-Mukai fourfolds}

From Theorem \ref{matrix} and Proposition \ref{prop_prim_action}, we deduce: 

\begin{lemma}
The kernel $E_0$ of the multiplication by $h$ in the small quantum cohomology ring is $E_0=\langle \alpha,\beta \rangle \oplus H^\bullet(X)^{\mathrm{prim}}$, where 
$$\alpha=2\sigma_2-3\sigma_{11}-2q, \qquad \beta=2\pt-2q\sigma_2-4q^2.$$
\end{lemma}

We will denote $E_0^{\mathrm{amb}}:= E_0 \cap H^\bullet(X)^{\mathrm{amb}}=\langle \alpha, \beta\rangle$. 

\medskip
We will consider a particular deformation of the small quantum cohomology ring inside the big quantum cohomology ring, along  the degree $(-1)$ variable $t$ dual to $\sigma_2$. In other words, we consider $(H^\bullet(X)[q][[t]], \star_t)$, where:
$$
\sigma_a\star_t \sigma_b=\sigma_a\cup \sigma_b + \sum_{\sigma_c,d,e}\langle \sigma_a,\sigma_b,\sigma_c^\vee,(\sigma_2)^e \rangle_d \frac{q^d t^e}{e!} \sigma_c.
$$
The Euler field in this deformation is  $\mathrm{Eu}_{t\sigma_2}=2h-t\sigma_2$. 
Let us compute  the matrix of the multiplication operator by this Euler field, at first order in $t$. 

\begin{lemma}
\label{lem_comp_def}
For $\alpha\in H^\bullet(X)$ and  $d\geq 1$, the coefficient of $q^d t$ in $\mathrm{Eu}_{t\sigma_2}\star_t \alpha$ is 
$$
\sum_{\sigma_a} (2d-1)\langle \sigma_2, \alpha , \sigma_a^\vee \rangle_d \sigma_a.
$$
\end{lemma}

\begin{proof}
The multiplication by $2h$ plus the divisor axiom gives the factor $2d$, while the factor $-1$ comes from the multiplication by $-t\sigma_2$.
\end{proof}

From the formulas we obtained for the quantum products by $h$ (Theorem \ref{matrix}) and $\sigma_2$ (``Conclusion'' of Sec. \ref{sec_ambient}), we deduce that the quantum multiplication by $\mathrm{Eu}_{t\sigma_2}$ is given, in the basis $1,\sigma_1,\sigma_2,\sigma_{11},\sigma_3,\sigma_{31}=2\pt$, and at first order in $t$, by the matrix:

$$\begin{pmatrix}
 0&12q&240q^2t&156 q^2t&48 q^2& 880q^3t \\
 2&10qt&20q&12q&180q^2t&48q^2 \\
 -t&2&8qt&8qt&8q&84q^2t \\
 0&2&12qt&4qt&4q& 72q^2t \\
 0&-3t&6&4&10qt&12q \\
 0&0&-2t&-t&2&0 
 \end{pmatrix}.$$

 \begin{lemma}
 \label{lem_Jordan_block}
Let  $\beta(t):=2\pt-2q\sigma_2-4q^2-16q^2t\sigma_1 -4qt\sigma_3$. Then
 $$\mathrm{Eu}_{t\sigma_2}\star_t \alpha = -4qt\alpha-t\beta(t)+O(t^2),$$
 $$\mathrm{Eu}_{t\sigma_2}\star_t \beta(t) = -4qt\beta(t)+O(t^2).$$
 \end{lemma}
 
 
 \begin{lemma}
 \label{lem_eigenvalues_primitive}
 For any $\gamma\in H^\bullet(X)^{\mathrm{prim}}$,  we have  $$\mathrm{Eu}_{t\sigma_2}\star_t \gamma=(-4qt+O(t^2))\gamma.$$
 \end{lemma}
 
 \begin{proof}
 The product $\mathrm{Eu}_{t\sigma_2}\star_t \gamma$ is computed as a sum in which the only Gromov-Witten invariants that appear are of the form $\langle \gamma,\sigma_2,\ldots ,\sigma_2,\beta\rangle_d$ for $d\geq 0$, $\beta$ a cohomology class of the right degree, and a certain number of appearances of $\sigma_2$. As in the proof of Proposition \ref{prop_prim_action}, if $\beta$ is an ambient class, the monodromy invariance of Gromov-Witten invariants implies that $\langle \gamma,\sigma_2,\ldots ,\sigma_2,\beta\rangle_d=0$. This implies that $\mathrm{Eu}_{t\sigma_2}\star_t \gamma$ belongs to
 $H^\bullet(X)^{\mathrm{prim}}[[q,t]]$. 
 
 Since the monodromy action on the primitive cohomology is irreducible, the action of the monodromy invariant multiplication by $\mathrm{Eu}_{t\sigma_2}$ must be an homothety  $\lambda(t)\mathrm{id}$. Then there must exist an atom $\boldsymbol\lambda$ whose associated eigenvalue at a general point $\xi$ is (a deformation of) $\lambda(t)$, and such that $\nu_{\xi,\lambda(t)}^X\neq 0$. On the other hand, we know that there exist Noether-Lefschetz divisors in the moduli space of Gushel-Mukai fourfolds, that parametrize rational manifolds. On these divisors $\rho_{\xi,\lambda(t)}^X\ge 3$,
 and since they have only one extra Hodge class,  this implies  that on the very general Gushel-Mukai fourfold  $\rho_{\xi,\lambda(t)}^X\geq 2$. 
 
 But the only eigenvalue of the action of $\mathrm{Eu}_{t\sigma_2}$ on the ambient cohomology having multiplicity at least two  is $-4qt+O(t^2)$. So  this eigenvalue must coincide with $\lambda(t)$, and the Lemma follows.
 \end{proof}
 
 Denote by $\boldsymbol\lambda$ the atom appearing in the proof of the previous Lemma, with eigenvalue $\lambda_0(t)=-4qt+O(t^2)$  at the point $\xi_0=t\sigma_2$. The generalized eigenspace of $\lambda_0(t)$ is $$E_{t\sigma_2,\lambda_0(t)}=\langle \tilde{\alpha}(t),\tilde{\beta}(t)\rangle \oplus H^\bullet(X)^{\mathrm{prim}},$$ where 
 $\tilde{\alpha}(t), \tilde{\beta}(t)\in H^\bullet(X)^{\mathrm{amb}}$ are suitable deformations of $\alpha, \beta(t)$ respectively. 
 
 \begin{lemma}
 \label{lem_Jblock}
  We have $\mathrm{rank}_\CC(\kappa_{t\sigma_2}-\lambda_0(t))|_{E_{t\sigma_2,\lambda_0(t)}}=1$. The unique Jordan block of size two of $(\kappa_{t\sigma_2}-\lambda_0(t))|_{E_{t\sigma_2,\lambda_0(t)}}$ is obtained by restricting $\kappa_{t\sigma_2}-\lambda_0(t)$ to $E_{t\sigma_2,\lambda_0(t)}\cap H^\bullet(X)^{\mathrm{amb}}=\langle \tilde{\alpha}(t),\tilde{\beta}(t)\rangle $.
 \end{lemma}
 
 \begin{proof}
 In the proof of Lemma \ref{lem_eigenvalues_primitive} we have seen that $\rho^X_{t\sigma_2,\lambda_0(t)}$ is  greater or equal to two (because it is so on a generic deformation $\xi$ of $\xi_0=t\sigma_2$). Moreover, $\kappa_{t\sigma_2}|_{H^\bullet(X)^{\mathrm{prim}}}=\lambda_0(t)\mathrm{id}$, and  $\mathrm{rank}_\CC(\kappa_{t\sigma_2}-\lambda_0(t))|_{E_{t\sigma_2,\lambda_0(t)}}\geq 1$ by Lemma \ref{lem_Jordan_block}. 
 
 On the other hand, the image of  $H^\bullet(X)^{\mathrm{amb}}$ by the action of  $\kappa_{t\sigma_2}-\lambda_0(t)$ is contained in $H^\bullet(X)^{\mathrm{amb}}[[q,t]]\cap E_{t\sigma_2,\lambda_0(t)}$. Indeed, in the definition of this operator,  only Gromov-Witten invariants with at most one primitive insertion appear. By the irreducibility of the monodromy action, those with exactly one primitive insertion have to vanish. But these Gromov-Witten invariants are those that give the factor in the direction of the primitive cohomology. Thus $(\kappa_{t\sigma_2}-\lambda_0(t))(\langle \tilde{\alpha}(t),\tilde{\beta}(t)\rangle)\subset \langle \tilde{\alpha}(t),\tilde{\beta}(t)\rangle$, which implies the statement.
 \end{proof}
 
We use the same notations as in the Remark after the proof of Theorem \ref{thm_irrationality_GM}. For $\zeta_1,\zeta_2\in D(0,1)$,  define the $\mathbb{Q}$-evaluation function 
$$
\mathrm{ev}:\left\{q,t_0,\dots t_5 \right\} \to S^*, 
$$
$$
q\mapsto \zeta_1 b^4, \quad t_k\mapsto 0 \quad \forall k\neq 2, \quad t_2=:t\mapsto \zeta_2 b^{-2}.
$$

By \cite[Proposition 21]{jg}, this defines a $\mathbb{Q}$-evaluation map 
$$\mathrm{ev}:\hat{R}^*(X,\mathbb{Q})=\mathbb{Q}[[q,t_0,\dots,t_5]]\to S^*,$$ 
such that the evaluation $\mathrm{ev}(\kappa_\tau)$ is just the multiplication by $\mathrm{Eu}_{t\sigma_2}(X)=2h-t\sigma_2$  acting on the deformation of the small quantum cohomology ring along the direction $t$. We will take $t,\zeta_2\in \QQ$. By definition $\mathrm{ev}(\kappa_\tau)$ is a matrix with coefficients in $S^*=F[b^{\pm 1}]$ (the variable $b$ just keeps track of degrees).
 \begin{lemma}
 \label{lem_Q_eigenvalue}
 $\mathrm{ev}(\lambda_0(t))\in S^*$ and $\mathrm{ev}(\tilde{\alpha}(t)),\mathrm{ev}(\tilde{\beta}(t))\in H^\bullet(X)^{\mathrm{amb}}_{S^*}$.
 \end{lemma}
 
 \begin{proof}
 The eigenvalue $\mathrm{ev}(\lambda_0(t))$ of the matrix $\mathrm{ev}(\kappa_\tau)$ is the only one with multiplicity $24$, so it must be fixed by the Galois group of the decomposition field of the characteristic polynomial of $\mathrm{ev}(\kappa_\tau)$. Since we are in characteristic zero, this implies that $\mathrm{ev}(\lambda_0(t))\in S^*$. The statement about the elements $\mathrm{ev}(\tilde{\alpha}(t)),\mathrm{ev}(\tilde{\beta}(t))$ follows by Gauss' algorithm.
 \end{proof}
 
 \begin{theorem}
 \label{thm_K3&irrationality_GM}
 If $X$ is a rational Gushel-Mukai fourfold, there exists a projective K3 surface $S$ and an isomorphism of rational Hodge structures
 $$
 H^4(X,\mathbb{Q})^{\mathrm{prim}}\simeq H^2(S,\mathbb{Q})(-1).
 $$
 \end{theorem}

 Of course there is a huge difference between this statement over $\QQ$ and the same statement over $\ZZ$. By definition, $H^2(S,\QQ)$ and its Hodge structure determine the K3 surface $S$ only up to a rational isogeny. It was proved in \cite{isogenous} that for two projective K3 surfaces, being isogenous is equivalent to having the same Chow motive, or to  be joined by a chain of twisted derived equivalences. A typical example is another
 K3 surface $S'$ that is a fine moduli space of stable sheaves on $S$. 
 
 One could wonder if some
 theory of atoms over $\ZZ$ could allow to upgrade our theorem to integral coefficients. 
  
 \begin{proof}[Proof, following \cite{kkpy} and \cite{jg}]
As in the proof of Theorem \ref{thm_irrationality_GM}, we will use the fact that the decomposition in generalized eigenspaces for the operator $\kappa_{\xi}$ of quantum multiplication  by $\mathrm{Eu}_{\xi}$ is finer, at a generic point $\xi\in H^\bullet(X)^{\mathrm{Hdg}}$, than at a special point $\xi_0$. Our special point will be $\xi_0=t\sigma_2$, so that $\mathrm{Eu}_{\xi_0}=\mathrm{Eu}_{t\sigma_2}=2h-t\sigma_2$ acting on the deformation along $t$ of the small quantum cohomology. We work modulo $t^2$. Let $u$ be an eigenvalue of $\kappa_{\xi_0}$, and $E_{\xi_0,u}$ be the corresponding generalized eigenspace. Recall that $\nu_{\xi_0,u}^X:=\dim_\CC(E_{\xi_0,u}\cap H^{(2)}(X))$, and let also  $$\nu'^X_{\xi_0,u}:=\dim_\CC(E_{\xi_0,u}\cap H^{(1)}(X))\quad \mathrm{and}\quad  \gamma_{\xi_0,u}^X:=\mathrm{rank}_\CC(\kappa_{\xi_0}-u)|_{E_{\xi_0,u}},$$ for generic $q,t$.
In our case, by Lemma \ref{lem_Jblock}, we have  $$\nu^X_{\xi_0,\lambda_0(t)}=1, \qquad  \nu'^X_{\xi_0,\lambda_0(t)}=0, \qquad \gamma^X_{\xi_0,\lambda_0(t)}=1.$$ 
Moreover the matrix of $\kappa_{\xi_0}$, restricted to $E_{\xi_0,\lambda_0(t)}^{\mathrm{amb}}$ is given  at first order in $t$ by Lemma \ref{lem_Jordan_block}.
By Lemma \ref{lem_Jblock}
$$\ker(\kappa_{\xi_0}-\lambda_0(t) \mathrm{id})=H^\bullet(X)^{\mathrm{prim}} \oplus \langle \tilde{\beta}(t)\rangle\quad\mathrm{and}\quad  \mathrm{im}(\kappa_{\xi_0}-\lambda_0(t) \mathrm{id})=\langle \tilde{\beta}(t)\rangle .$$ 

We use the  evaluation map $\mathrm{ev}$ defined before Lemma \ref{lem_Q_eigenvalue}. Since $$\nu^X_{\xi_0,\lambda_0(t)}=1, \qquad \nu'^X_{\xi_0,\lambda_0(t)}=0 \qquad  \gamma^X_{\xi_0,\lambda_0(t)}=1,$$ 
property $\varheartsuit_{\hat{R}^*(X,\tilde{\mathbb{Q}}_{\mathrm{ext}})}$ (\cite[Definition 27]{jg}) is not satisfied for our evaluation map. This implies by \cite[Remark 29]{jg} that property $\varheartsuit_{\hat{R}^*(X,R)}$ is not satisfied either for any number field $R$ containing $\tilde{\mathbb{Q}}_{\mathrm{ext}}$. Moreover the above equalities for the kernel and image of $\kappa_{\xi_0}$ are also valid when applying $\mathrm{ev}$:
$$
\ker(\mathrm{ev}(\kappa_{\tau}-\lambda_0(t) \mathrm{id}))=\mathrm{ev}(H^\bullet(X)^{\mathrm{prim}} \oplus \langle \tilde{\beta}(t)\rangle)\quad\mathrm{and}$$
$$\quad \mathrm{im}(\mathrm{ev}(\kappa_{\tau}-\lambda_0(t) \mathrm{id}))=\mathrm{ev}(\langle \tilde{\beta}(t)\rangle ).
$$

If $X$ is rational,
arguing as in the proof of \cite[Theorem 56]{jg}, we deduce that the quantum cohomology of some polarized K3 surface $S$ must be sent isomorphically to $\mathrm{ev}(E_{\xi_0,\lambda_0(t)})$. Let us denote by $f$ such an isomorphism.
 
 By \cite[Corollary 51]{jg}, in the quantum cohomology of the K3 surface $S$, if the rank of the multiplication $\kappa^{S}_{\nu}$ by the Euler vector field is nonzero for some $\nu\in H^\bullet(S)^{\mathrm{Hdg}}$, then  $$\ker(\kappa^S_\nu)=H^2(S)\oplus H^4(S)\quad\mathrm{and}\quad
 \mathrm{im}(\kappa^S_\nu)=H^4(S).$$ Arguing again as in the proof of \cite[Theorem 56]{jg}, this implies that $f$ sends $\kappa_{\xi_0}$ to $\kappa^S_\nu$,  and the image and kernel of these endomorphisms are sent one to the other (more precisely, there exists an evaluation map $\mathrm{ev}^{S}$ of the quantum cohomology of the K3 surface such that $f$ sends $\mathrm{ev}(\kappa_\tau)$ to $\mathrm{ev}^{S}(\kappa^S_\tau)$). One deduces that, for a generic choice of $q,t\in\mathbb{Q}$, the composition $$H^\bullet(X)^{\mathrm{prim}}\subset E_{\xi_0,\lambda_0(t)}\xrightarrow{f} H^\bullet(S)\to H^2(S)$$ is an isomorphism of $\mathbb{Q}$-Hodge structures, which concludes the proof. 
 
 Notice that in the proof of \cite[Theorem 56]{jg}, in order to pass from $\tilde{Q}_{\mathrm{ext}}$-coefficients (where $\tilde{Q}_{\mathrm{ext}}$ is a field extension of $\QQ$ explicitly defined in \cite{jg}, allowing to use Iritani's blowup formula and containing all roots of characteristic polynomials of $\mathrm{ev}(\kappa_\tau)$ involved in all blowups of the weak factorization) one needs to use Lemma \cite[Lemma 57]{jg}. In this lemma the crucial point is that the basis which gives the Jordan form of $\mathrm{ev}(\kappa_\tau)$ has coefficients in $S^*=F[b^{\pm 1}]$ rather than $S^*_{\tilde{Q}_{\mathrm{ext}}}=S^*\otimes_\QQ \tilde{Q}_{\mathrm{ext}}$. In our situation, this is ensured 
 by Lemma \ref{lem_Q_eigenvalue}.
 \end{proof}



\begin{thebibliography}{Aa}

\bibitem{addington-thomas} 
Addington N.,  Thomas R.,
{\it Hodge theory and derived categories of cubic fourfolds}, 
Duke Math. J. {\bf 163} (2014), no. 10, 1885–1927. 

\bibitem{bvb}
B\"ohning Ch., Graf von Bothmer H.-Ch., 
{\it
Degenerations of Gushel-Mukai fourfolds, with a view towards irrationality proofs},
Eur. J. Math. {\bf 4} (2018), no. 3, 802–826. 

\bibitem{dim0} Debarre O., Iliev A., Manivel L., {\it
On nodal prime Fano threefolds of degree 10}, 
Sci. China Math. 54 (2011), no. 8, 1591–1609. 

\bibitem{dim} Debarre O., Iliev A., Manivel L., {\it Special prime Fano fourfolds of degree 10 and index 2}, in Recent Advances in Algebraic Geometry, pp. 123–155, London Math. Soc. Lecture Note Ser. {\bf 47}, Cambridge 2015.

\bibitem{debkuz} Debarre O.,  Kuznetsov A.,
{\it Gushel-Mukai varieties: linear spaces and periods}, 
Kyoto J. Math. {\bf 59} (2019), no. 4, 897–953.

\bibitem{debkuz2} Debarre O.,  Kuznetsov A.,
{\it  Quadrics on Gushel-Mukai varieties},  arXiv:2409.03528.

\bibitem{grz}
Grzelakowski K., Rampazzo M., Zhang S., {\it
Categorical resolutions and birational geometry of nodal Gushel-Mukai varieties},
arXiv:2602.14109.

\bibitem{jg}
Guéré J., {\it  On the irrationality of cubic fourfolds}, arXiv:2603.04518.

\bibitem{glz}
Guo H., Liu Z., Zhang S., {\it  Conics on Gushel-Mukai fourfolds, EPW sextics and Bridgeland moduli spaces}, Math. Res. Lett. {\bf 31} (2024), no. 4, 1061–1106.

\bibitem{hs}
Hoff M., Staglian\`o G., {\it New examples of rational Gushel-Mukai fourfolds}, Math. Z. {\bf 296}, 1585–1591 (2020).

\bibitem{isogenous}
Huybrechts D., {\it 
Motives of isogenous K3 surfaces}, 
Comment. Math. Helv. {\bf 94} (2019), no. 3, 445–458. 

\bibitem{im} Iliev A., Manivel L., 
{\it Fano manifolds of degree ten and EPW sextics},
Ann. Sci. Éc. Norm. Supér. (4) {\bf 44} (2011), no. 3, 393–426. 
\bibitem{Ir2} Iritani, H.
{\it Notes on the decomposition theorem for blowups},
arXiv:2604.10028.

\bibitem{kkpy} Katzarkov L., Kontsevich M., Pantev T., 
Yu T.Y.,  {\it Birational Invariants from Hodge Structures and Quantum Multiplication}, arXiv:2508.05105.

\bibitem{kp} Kuznetsov A.,  Perry A.,
{\it Derived categories of Gushel-Mukai varieties},
Compos. Math. {\bf 154} (2018), no. 7, 1362–1406. 

\bibitem{ppz} Perry A., Pertusi L., Zhao X.,
{\it Stability conditions and moduli spaces for Kuznetsov components of Gushel-Mukai varieties}, 
Geom. Topol. {\bf 26} (2022), no. 7, 3055–3121. 

\bibitem{schreieder} Schreieder S., {\it Quadric surface bundles over surfaces and stable rationality},  Algebra Number Theory {\bf 12} (2018), no. 2, 479–490.

\bibitem{sta}
Staglian\`o G., {\it 
Some new rational Gushel fourfolds}, 
Mediterr. J. Math. {\bf 18} (2021), no. 5, Paper No. 182, 17 pp. 

\bibitem{Vo}
Voisin C., {\it 
Hodge theory and complex algebraic geometry. II}, 
Mediterr. Cambridge Studies in Advanced Mathematics {\bf 77} (2003), pp. x+351

\end{thebibliography}
\end{document}